\documentclass[11pt]{amsart}
\usepackage[top=1.0in,bottom=1.0in,left=1.0in,right=1.0in]{geometry}
\usepackage[T1]{fontenc}
\usepackage[utf8]{inputenc}
\usepackage{lmodern}
\usepackage{amsmath,amsthm,amssymb,mathtools}
\usepackage{enumitem}
\usepackage{xcolor}
\usepackage[backref=page]{hyperref}
\usepackage{microtype}
\usepackage{mathrsfs}
\definecolor{ash}{RGB}{120,135,125}
\hypersetup{colorlinks=true, linkcolor=blue, citecolor=blue}
\usepackage[capitalise]{cleveref}
\newtheorem{theorem}{Theorem}[section]

\newtheorem{proposition}[theorem]{Proposition}
\newtheorem{corollary}[theorem]{Corollary}
\newtheorem{remark}[theorem]{Remark}
\newtheorem{question}[theorem]{Question}
\theoremstyle{definition}

\newcommand{\Gam}{\Gamma}
\newcommand{\Cay}{\operatorname{Cay}}
\newcommand{\ZZ}{\mathbb{Z}}
\newcommand{\RR}{\mathbb{R}}

\newcommand{\norm}[1]{\left\|#1\right\|}
\newcommand{\E}{\mathbb{E}}

\newcommand{\floor}[1]{\left\lfloor #1\right\rfloor}
\newcommand{\ceil}[1]{\left\lceil #1\right\rceil}

\newcommand{\doi}[1]{\href{https://doi.org/#1}{\textcolor{ash}{doi:\,#1}}}
\newcommand{\mr}[1]{\href{https://mathscinet.ams.org/mathscinet-getitem?mr=#1}{\textcolor{ash}{MR#1}}}

\title[Fast robbers on abelian Cayley graphs and digraphs]{Fast robbers on abelian Cayley graphs and digraphs}
\author{Arindam Biswas}
\address{}
\email{arin.math@gmail.com}
\thanks{}



\begin{document}

\begin{abstract}
We study the fast-robber version of the Cops and Robbers game on finite strongly
connected abelian Cayley digraphs, including undirected Cayley graphs as the symmetric case. For bounded out-degree $D$, we show that
$c_{1,\infty}\left(\Gamma\right)=O_D\left(n^{1-\frac{1}{D}}\right)$; in the
undirected case with $D\ge2$, this improves to the optimal exponent
$1-\frac{1}{\left\lfloor \frac{D}{2}\right\rfloor}$.
We also establish the degree-independent bound
$c_{1,\infty}\left(\Gamma\right)=O\!\left(\frac{n\left(\log\log n\right)^2}{\log n}\right)$.
These estimates follow from an optimized character-theoretic cyclic sweep over
subgroup quotients.
\end{abstract}

\maketitle

\section{Introduction}

The game of Cops and Robbers is a pursuit--evasion game on a finite connected
graph. A team of cops and a single robber occupy vertices and move alternately
along edges with speed at most one; the cops win when one of them reaches the
robber's vertex. The game was introduced independently by Nowakowski and
Winkler \cite{NowakowskiWinklerVertexToVertexPursuitInGraph} and Quilliot
\cite{QuilliotThesis}. Its basic parameter is the \emph{cop number}
$c\left(\Gamma\right)$, introduced by Aigner and Fromme
\cite{AignerFrommeGameOfCopsRobbers}, defined as the minimum number of cops
that guarantee capture on $\Gamma$. The unit-speed game has been studied from
structural, extremal, probabilistic, geometric, and algorithmic perspectives;
see the monograph of Bonato and Nowakowski
\cite{BonatoNowakowskiSTMLCopsRobbers}.

One of the main open problems in the ordinary game is Meyniel's conjecture \cite{FranklLargeGirthCayley}, which predicts that every connected $n$-vertex graph satisfies
$
  c\left(\Gamma\right)=O\left(\sqrt n\right).
$
For abelian Cayley graphs, the conjecture was established by Bradshaw
\cite[Theorem 1.1]{BradshawProofOfMeynielConjAbelianCayleyGraphs} and by
Bradshaw--Hosseini--Turcotte
\cite[Theorem 3.8]{BradshawHosseiniTurcotteCopsRobbersDirUnditAbelianCayley}.
In a different direction, Frankl proved the sharp degree-sensitive bound
$
  c\left(\Cay\left(G,S\right)\right)\le \left\lceil\frac{\left|S\right|+1}{2}\right\rceil
$
for connected undirected abelian Cayley graphs \cite{FranklPursuitGameOnCayleyGraph}. Thus even in the classical unit-speed game, abelian Cayley graphs already exhibit an interplay between order-sensitive and degree-sensitive estimates.

Bounded-degree graphs retain much of the difficulty of Meyniel-type problems.
For instance, Hosseini, Mohar, and Gonz\'alez Hermosillo de la Maza showed
that Meyniel's conjecture for subcubic graphs would imply the weak Meyniel
conjecture for all graphs \cite{HosseiniMoharGonzalezBoundedDegree}. This
motivates determining the exact exponent in the bounded-degree regime.

In this article we consider the \emph{fast-robber} variant, in which the cops have speed $1$ and the
robber may traverse a directed path of length at most $s\ge 1$ in one turn,
provided that the path avoids vertices occupied by cops. More generally,
$c_{r,s}\left(\Gamma\right)$ denotes the minimum number of speed-$r$ cops
required to capture a speed-$s$ robber. We write
$c_{1,\infty}\left(\Gamma\right)$ when the robber may traverse an arbitrary
finite cop-free directed path in one turn. In the undirected case, paths may
traverse edges in either direction. This model was studied systematically by
Fomin--Golovach--Kratochv{\'\i}l--Nisse--Suchan
\cite{FominGolovachKratochvilNisseSuchanPursuingFastRobber}, and subsequently
by Alon--Mehrabian \cite{AlonMehrabian}, Frieze--Krivelevich--Loh
\cite{FriezeKrivelevichLohVariationsOnCopsRobbers}, Mehrabian
\cite{Mehrabian}, and others.

The work of \cite{FominGolovachKratochvilNisseSuchanPursuingFastRobber} connected fast-robber pursuit to graph decompositions and width parameters. Frieze, Krivelevich, and Loh proved that, for each fixed finite robber speed $s$, every connected undirected $n$-vertex graph satisfies
$$
  c_{1,s}\left(\Gamma\right)
  \le
  \frac{n}{\alpha^{\left(1-o\left(1\right)\right)\sqrt{\log_\alpha n}}},
  \qquad \alpha=1+\frac1s,
$$
and they constructed graphs with $c_{1,\infty}\left(\Gamma\right)=\Omega\left(n\right)$ \cite[Theorems 1.3 and 1.4]{FriezeKrivelevichLohVariationsOnCopsRobbers}. Thus no sublinear upper bound is possible for an unbounded-speed robber on arbitrary graphs.

Mehrabian conjectured that for every fixed $s\ge 1$, every connected $n$-vertex graph satisfies $c_{1,s}\left(\Gamma\right)=O\left(n^{\frac{s}{s+1}}\right)$ \cite{Mehrabian}. Alon and Mehrabian proved that this exponent would be best possible: for every fixed $s\ge 1$ there exist connected $n$-vertex graphs with
$
  c_{1,s}\left(\Gamma\right)=\Omega\left(n^{\frac{s}{s+1}}\right).
$
For certain orders, their extremal construction is an abelian Cayley graph on
$\mathbb Z_2^{\,1+k\left(s+1\right)}$ of degree $2^k$. Their construction for
arbitrary orders is obtained by adjoining a path, an operation that does not
preserve the Cayley property \cite[Section 2]{AlonMehrabian}. Since an
undirected graph may be viewed as a symmetric digraph, the same lower bound
also applies in the directed setting. Balister, Bollob\'as, Narayanan, and Shaw
proved a different lower bound on the $N\times N$ grid, which has $n=N^2$
vertices: when the robber speed exceeds a sufficiently large absolute
constant, the required number of cops is
$\exp\!\left(\Omega\!\left(\frac{\log n}{\log\log n}\right)\right)$, which is
superpolylogarithmic but subpolynomial
\cite{BalisterBollobasNarayananShaw2017}. When both players have the same
speed, the fast game admits an exact reduction to the ordinary game. Mehrabian
proved for undirected graphs that
$
  c_{s,s}\left(\Gamma\right)=c\left(\Gamma^{\left[s\right]}\right),
$
where $\Gamma^{\left[s\right]}$ is obtained by joining pairs of vertices at
distance at most $s$ in $\Gamma$
\cite[Theorem 6.1]{MehrabianFastRobberExpandersRandomGraphs}. The same argument,
with directed distance, gives the corresponding identity for every finite
digraph; see Section~\ref{sec:symmetric}. The remaining case is asymmetric
pursuit, in which only the robber is accelerated. In the dense regime, the
domination bound from \cite[Section 4]{MehrabianFastRobberExpandersRandomGraphs},
together with the Arnautov--Payan estimate, gives a general upper bound. In the
sparse regime, the problem remains largely open: the conjectured fixed-speed
exponent $\frac{s}{s+1}$ tends to $1$ as $s\to\infty$ and therefore yields no
sublinear estimate uniform in the robber speed. To the best of the author's
knowledge, comparatively little is known for digraphs when the robber has
speed $s>1$.

For a fast robber, dependence on the degree is necessary for any polynomial
improvement over the trivial linear bound. Indeed, the weak Meyniel conjecture
fails for the unbounded-speed game even within the class of abelian Cayley
graphs: by Remark~\ref{rem:near-tight} there are connected abelian
Cayley graphs of degree $O\left(\log n\right)$ with
$
  c_{1,\infty}\left(\Gam\right)=\Omega\!\left(\frac{n}{\log n}\right),
$
and $n^{1-\varepsilon}=o\!\left(\frac n{\log n}\right)$ for every
$\varepsilon>0$. Hence no bound of the form
$c_{1,\infty}\left(\Gam\right)=O\left(n^{1-\varepsilon}\right)$ can hold
uniformly in the degree. The bound in
Theorem~\ref{thm:intro-degree-free} is therefore optimal up to a factor of
$\left(\log\log n\right)^2$. Any polynomial improvement must be
degree-sensitive; Theorem~\ref{thm:intro-bounded} provides such an improvement
and gives the optimal exponent in the undirected case.

The principal objectives of this article are:
\begin{enumerate}
  \item to formulate the fast-robber game on digraphs and establish quotient-lifting tools and an equal-speed reduction;
  \item to prove an optimized quotient-sweeping bound for asymmetric pursuit that is uniform in the robber's speed;
  \item to deduce a degree-independent sublinear bound for finite abelian Cayley digraphs and obtain sharp exponents for undirected graphs of bounded degree; and
  \item to determine the asymptotic order of the fast-robber cop number of the Alon--Mehrabian Cayley graphs, including for an unbounded-speed robber.
\end{enumerate}

The proofs use character theory to construct a cyclic quotient in which every
generator has bounded cyclic displacement. A moving full-fibre barrier and a stationary
full-fibre barrier then confine the robber. Since the strategy is independent
of the length of the robber's path, it also applies to an unbounded-speed
robber.

\subsection{Main results}
The equal-speed reduction yields the following degree-sensitive estimates.

\begin{theorem}[Equal speeds]\label{thm:intro-equal}
Let $\Gamma=\Cay\left(G,S\right)$ be a strongly connected abelian Cayley
digraph of out-degree $d$ and order $n$. Then, for every $s\ge 1$,
$$
  c_{s,s}\left(\Gamma\right)
  = c\left(\Gamma^{\left[s\right]}\right)
  \le
  \min\!\left\{
    \sqrt{\frac{2}{e\left(\sqrt2-1\right)}}\sqrt n+2,\,
    \binom{d+s}{s}
  \right\}.
$$
If $S=-S$, the right-hand side may be replaced by
$$
  \min\!\left\{
    \frac{1}{\sqrt{e\left(\sqrt2-1\right)}}\sqrt n+\frac72,\,
    \left\lceil\frac12\binom{d+s}{s}\right\rceil
  \right\}.
$$
\end{theorem}

For asymmetric pursuit, optimizing the character argument over all subgroup
quotients gives the following estimate.

\begin{theorem}[Optimized quotient bound]\label{thm:intro-quotient}
Let $\Gamma=\Cay\left(G,S\right)$ be a strongly connected abelian Cayley digraph of out-degree $d$ and order $n$. For every proper subgroup $K<G$, define
$$
  r_K:=\left|\left\{\left\{\bar g,-\bar g\right\}:g\in S,\ \bar g\ne0
  \text{ in }G/K\right\}\right|,
  \qquad
  N_K:=\floor{\left(\left|G/K\right|-1\right)^{\frac{1}{r_K}}},
$$
where $\bar g:=g+K$, and let $N_*:=\max_{K<G}N_K$. Then, for every $s\ge1$,
$$
  c_{1,s}\left(\Gamma\right)\le c_{1,\infty}\left(\Gamma\right)
  \le
  \min\!\left\{
    \frac{n\left(1+\log\left(d+1\right)\right)}{d+1},\,
    \frac{2n}{N_*}
  \right\}.
$$
\end{theorem}

\begin{theorem}[Bounded degree]\label{thm:intro-bounded}
Fix $D\ge1$. Every strongly connected abelian Cayley digraph $\Gamma=\Cay\left(G,S\right)$ of order $n$ with $\left|S\right|\le D$ satisfies, for every $s\ge1$,
$$
  c_{1,s}\left(\Gamma\right)\le c_{1,\infty}\left(\Gamma\right)
  =O_D\left(n^{1-\frac{1}{D}}\right).
$$
If $S=-S$ and $D\ge2$, this improves to
$
  O_D\left(n^{1-\frac{1}{\left\lfloor \frac{D}{2}\right\rfloor}}\right),
$
and this exponent is best possible: by Section~\ref{sec:conclusion} there are
connected undirected abelian Cayley graphs of degree at most $D$, of
arbitrarily large order, with
$
  c_{1,\infty}\left(\Gam\right)
  =\Omega_D\!\left(n^{1-\frac{1}{\left\lfloor \frac D2\right\rfloor}}\right).
$
\end{theorem}

\begin{theorem}[Degree-independent bound]\label{thm:intro-degree-free}
Every strongly connected abelian Cayley digraph $\Gamma$ of order $n$ satisfies, for every $s\ge1$,
$$
  c_{1,s}\left(\Gamma\right)\le c_{1,\infty}\left(\Gamma\right)
  =O\!\left(\frac{n\left(\log\log n\right)^2}{\log n}\right).
$$
\end{theorem}

\begin{theorem}[Alon--Mehrabian examples]\label{thm:intro-alon}
Fix $s\ge1$. For each sufficiently large $k$, let $\Gamma=\Gamma_{s,k}$ be the abelian Cayley graph constructed by Alon and Mehrabian \cite[Section 2]{AlonMehrabian}, with
$$
  G=\ZZ_2^{\,1+k\left(s+1\right)},\qquad d=2^k,\qquad n=2d^{s+1}.
$$
Then
$$
  c_{1,\infty}\left(\Gamma\right)\le\frac nd
  =2d^s
  =2^{\frac{1}{s+1}}n^{\frac{s}{s+1}}.
$$
Consequently,
$$
  c_{1,s}\left(\Gamma\right)=\Theta_s\!\left(n^{\frac{s}{s+1}}\right),
  \qquad
  c_{1,\infty}\left(\Gamma\right)=\Theta_s\!\left(n^{\frac{s}{s+1}}\right),
$$
as $k\to\infty$.
\end{theorem}

The proof of Theorem~\ref{thm:intro-alon} appears in
Section~\ref{sec:alon} and establishes the conjectured order for this
extremal Cayley family even when the robber has unbounded speed.
Theorem~\ref{thm:intro-bounded} does not follow from sparsity alone:
Mehrabian's isoperimetric lower bound shows that every bounded-degree
undirected expander family satisfies
$c_{1,\infty}\left(\Gamma\right)=\Theta\left(n\right)$, and a random
$d$-regular graph asymptotically almost surely has
$c_{1,\infty}\left(\Gamma\right)=\Theta\left(n\right)$ for each fixed
$d\ge3$ \cite[Theorem 3.3 and Corollary 4.8]{MehrabianFastRobberExpandersRandomGraphs}.
The sublinear estimate therefore depends essentially on the abelian Cayley
structure. For abelian Cayley graphs, the
ordinary $O\left(\sqrt n\right)$ bound is sharp up to constant factors by the
constructions of Hasiri and Shinkar \cite{HasiriShinkarMeynielExtremal}.
Beyond the bounded-degree range, if
$
  d \le \frac{\log n}{\kappa\log\log n},
$
then
$
  c_{1,\infty}\left(\Gamma\right)=O\!\left(\frac{n}{\left(\log n\right)^{\kappa}}\right),
$
with the exponent $2\kappa$ in the undirected case; see
Corollary~\ref{cor:logdeg}. Combining this sparse estimate with domination in
the complementary dense regime gives Theorem~\ref{thm:intro-degree-free}.

In the Alon--Mehrabian family every generator is an involution. Hence the
subgroup generated by the involutory generators equals $G$ and cannot be
used as a proper quotient subgroup. The proof of
Theorem~\ref{thm:intro-alon} instead uses a noncyclic coordinate quotient.

The rest of the paper is organized as follows. Section~\ref{sec:prelim} introduces the notation and recalls the ordinary-game results used later. Section~\ref{sec:basic-bounds} treats the equal-speed reduction and domination. Section~\ref{sec:quotient-method} develops quotient lifting, character selection, and sweeping, and proves the optimized upper bound. Section~\ref{sec:consequences} gives bounded-degree and degree-independent consequences and concludes with examples, lower bounds, and questions.

\section{Preliminaries}\label{sec:prelim}

We use the standard asymptotic notation as $n\to\infty$. All logarithms are
natural unless a base is explicitly indicated. For nonnegative functions $f$
and $g$, we write $f\left(n\right)=O\left(g\left(n\right)\right)$ if there
are constants $C>0$ and $n_0$ such that
$f\left(n\right)\le Cg\left(n\right)$ for all $n\ge n_0$, and
$f\left(n\right)=\Omega\left(g\left(n\right)\right)$ if
$g\left(n\right)=O\left(f\left(n\right)\right)$. We write
$f\left(n\right)=\Theta\left(g\left(n\right)\right)$ when both
$f\left(n\right)=O\left(g\left(n\right)\right)$ and
$f\left(n\right)=\Omega\left(g\left(n\right)\right)$. A subscript records the
parameters on which the implicit constants may depend; for example,
$O_D\left(g\left(n\right)\right)$ means that $C$ and $n_0$ may depend on
$D$, but not on $n$. Throughout the paper, $\Gam=\left(V,E\right)$ denotes a finite strongly connected digraph; an undirected graph is identified with the symmetric digraph obtained by replacing each edge by two opposite arcs. A
\emph{directed walk} is a vertex sequence $v_0,v_1,\dots,v_\ell$ such that
$v_{i-1}\to v_i$ is an arc for every $1\le i\le\ell$; vertices may be
repeated. A \emph{directed path} is a directed walk with no repeated vertices.
In either case, the length is $\ell$, and paths of length zero are allowed.

In the $\left(r,s\right)$ game, the cops choose their initial vertices first,
and the robber then chooses its initial vertex. Thereafter the cops move first in
each round. A cop of speed $r$ may remain at its current vertex or move along a directed path of
length at most $r$. A robber of speed $s$ may remain at its current vertex or move along a directed
path of length at most $s$, with the restriction that no vertex of that path,
including its terminal vertex, is occupied by a cop. Capture occurs as soon as
a cop and the robber occupy the same vertex; since the robber may not move onto
an occupied vertex, this can happen only on a move by the cops. We write
$c_{r,s}\left(\Gam\right)$ for the least number of speed-$r$ cops that
guarantee capture of a speed-$s$ robber. In the unbounded-speed game the robber
may traverse any finite cop-free directed path in one turn; the corresponding
parameter is denoted by $c_{1,\infty}\left(\Gam\right)$. Monotonicity in the
robber's speed gives
$
  c_{1,s}\left(\Gam\right)\le c_{1,\infty}\left(\Gam\right)
$
for every finite $s$.

If $G$ is a finite abelian group written additively and $S\subseteq G\setminus\left\{0\right\}$, then the Cayley digraph $\Cay\left(G,S\right)$ has vertex set $G$ and an arc $x\to x+g$ for every $x\in G$ and $g\in S$. It is in- and out-regular of degree $d:=\left|S\right|$. Since $G$ is finite, the submonoid generated by $S$ is the subgroup generated by $S$: if $g\in S$ has order $q$, then $-g=\left(q-1\right)g$ is a nonnegative sum of elements of $S$. Consequently, $\Cay\left(G,S\right)$ is strongly connected if and only if $\left\langle S\right\rangle=G$. We consider nontrivial groups, so $\left|G\right|\ge2$. When $S=-S$, we regard $\Cay\left(G,S\right)$ as an undirected Cayley graph.

For a finite group $Q$, its exponent $\operatorname{exp}\left(Q\right)$ is
the least common multiple of the orders of its elements. Define
$$
  T:=\left\{g\in S:2g=0\right\},\qquad t:=\left|T\right|,\qquad H:=\left\langle T\right\rangle.
$$
Let $r$ be the number of inverse-pair classes meeting $S\setminus T$:
$$
  r:=\left|\left\{\left\{g,-g\right\}:g\in S\setminus T\right\}\right|.
$$
Then
$$
  \frac{d-t}{2}\le r\le d-t.
$$
The lower bound is attained when $S\setminus T$ is symmetric, whereas the
upper bound is attained when no two distinct elements of $S\setminus T$ are
negatives of one another.

For an undirected graph $\Gam$ on $n$ vertices and a nonempty set $A\subseteq V\left(\Gam\right)$, let
$$
  \partial_v A:=N\left(A\right)\setminus A
$$
be its external vertex boundary. The vertex-isoperimetric number of $\Gam$ is
$$
  \iota_v\left(\Gam\right):=
  \min_{0<\left|A\right|\le \frac n2}\frac{\left|\partial_v A\right|}{\left|A\right|}.
$$

For a digraph $\Gam$ and a positive integer $q$, let $\Gam^{\left[q\right]}$ denote the digraph on $V\left(\Gam\right)$ with an arc $x\to y$ whenever $x\ne y$ and the directed distance from $x$ to $y$ in $\Gam$ is at most $q$. If $\Gam=\Cay\left(G,S\right)$, then
$
  \Gam^{\left[q\right]} = \Cay\left(G,S^{\left[q\right]}\right),
$
where $S^{\left[q\right]}$ is the set of all nonzero group elements expressible as a sum of at most $q$ elements of $S$.

\section{Graph-power and domination bounds}\label{sec:basic-bounds}

\subsection{An equal-speed reduction}\label{sec:symmetric}

When both players have the same speed, the fast game reduces to the ordinary game on a graph power.

\begin{proposition}\label{prop:symmetric-general}
For every finite digraph $\Gam$ and every $s\ge1$,
$
  c_{s,s}\left(\Gam\right)=c\left(\Gam^{\left[s\right]}\right).
$
\end{proposition}

\begin{proof}
A move in $\Gam^{\left[s\right]}$ from $u$ to $v$ corresponds to a directed $u$--$v$ path of length at most $s$ in $\Gam$. Thus every cop move in $\Gam^{\left[s\right]}$ can be realized in the speed-$s$ game, and every legal robber move in the speed-$s$ game induces a legal move in $\Gam^{\left[s\right]}$. Consequently,
$
  c_{s,s}\left(\Gam\right)\le c\left(\Gam^{\left[s\right]}\right).
$

For the reverse inequality, let the cops simulate a winning speed-$s$ strategy
on $\Gam$ as long as every robber move in $\Gam^{\left[s\right]}$ is induced
by a cop-free directed path of length at most $s$ in $\Gam$. Suppose that the
robber first uses an additional arc $u\to v$ for which every directed
$u$--$v$ path of length at most $s$ meets the current cop set. The cops then
terminate the simulation and choose one such path
$u=w_0,w_1,\dots,w_\ell=v$ with $\ell\le s$. The robber occupies $u$, and its
move in $\Gam^{\left[s\right]}$ is legal, so neither $u$ nor $v$ is occupied;
hence some internal vertex $w=w_i$ with $1\le i\le\ell-1$ is occupied by a cop. The
terminal segment $w_i,\dots,w_\ell$ has length at most $s$ and $w\ne v$, so
$w\to v$ is an arc of $\Gam^{\left[s\right]}$ and that cop captures the robber
on the cops' next move. Hence
$
  c\left(\Gam^{\left[s\right]}\right)\le c_{s,s}\left(\Gam\right).
$
The two inequalities prove the identity, which is the directed analogue of
Mehrabian's argument for graphs.
\end{proof}

\begin{proof}[Proof of Theorem~\ref{thm:intro-equal}]
By Proposition~\ref{prop:symmetric-general},
$$
  c_{s,s}\left(\Gam\right)=c\left(\Gam^{\left[s\right]}\right).
$$
The power $\Gam^{\left[s\right]}$ is the strongly connected abelian Cayley digraph
$
  \Cay\left(G,S^{\left[s\right]}\right).
$
Applying the directed bounds of Bradshaw--Hosseini--Turcotte
\cite[Proposition 4.1 and Theorem 4.8]{BradshawHosseiniTurcotteCopsRobbersDirUnditAbelianCayley} gives
$$
  c\left(\Gam^{\left[s\right]}\right)
  \le
  \min\!\left\{
    \sqrt{\frac{2}{e\left(\sqrt2-1\right)}}\sqrt n+2,\,
    \left|S^{\left[s\right]}\right|+1
  \right\}.
$$
Since $G$ is abelian, every element
$x\in S^{\left[s\right]}\cup\left\{0\right\}$ has a representation
$$
  x=\sum_{u\in S}m_u u,
  \qquad
  m_u\in\ZZ_{\ge0},\quad \sum_{u\in S}m_u\le s.
$$
Introduce the slack variable
$m_0:=s-\sum_{u\in S}m_u$. The number of nonnegative integer solutions of
$
  m_0+\sum_{u\in S}m_u=s
$
is $\binom{d+s}{s}$. Different solutions may represent the same group
element, so this gives an upper bound; after excluding the zero element,
$
  \left|S^{\left[s\right]}\right|\le\binom{d+s}{s}-1.
$

If $S=-S$, then $S^{\left[s\right]}=-S^{\left[s\right]}$, so $\Gam^{\left[s\right]}$ is undirected. The sharper undirected order bound of Bradshaw--Hosseini--Turcotte \cite[Theorem 3.8]{BradshawHosseiniTurcotteCopsRobbersDirUnditAbelianCayley} and Frankl's bound \cite{FranklPursuitGameOnCayleyGraph} give
$$
  c\left(\Gam^{\left[s\right]}\right)
  \le
  \min\!\left\{
    \frac{1}{\sqrt{e\left(\sqrt2-1\right)}}\sqrt n+\frac72,\,
    \left\lceil\frac{\left|S^{\left[s\right]}\right|+1}{2}\right\rceil
  \right\},
$$
and the result follows from the estimates above.
\end{proof}

\subsection{A domination bound}\label{sec:domination}

The following domination argument is standard. Mehrabian applied it to the
unbounded-speed game \cite[Section 4]{MehrabianFastRobberExpandersRandomGraphs}
and combined it with the domination estimate of Arnautov \cite{Arnautov1974}
and Payan \cite{Payan1975}; see also Alon--Spencer
\cite{AlonSpencerProbMethod} and
\cite[Theorem 4.3]{MehrabianFastRobberExpandersRandomGraphs}.

\begin{proposition}\label{thm:domination}
Let $\Gam$ be a digraph on $n$ vertices with minimum in-degree at least $d\ge1$. Then, for every finite $s$,
$$
  c_{1,s}\left(\Gam\right)\le c_{1,\infty}\left(\Gam\right)
  \le \gamma^+\left(\Gam\right)
  \le \frac{n\left(1+\log\left(d+1\right)\right)}{d+1},
$$
where $\gamma^+\left(\Gam\right)$ is the minimum size of an out-dominating set: a set $D$ such that every vertex outside $D$ has an in-neighbour in $D$.
\end{proposition}

\begin{proof}
Let $D$ be an out-dominating set and place one cop on each vertex of $D$. If the robber starts at $v\notin D$, then some cop is at an in-neighbour of $v$ and captures it on the first move by the cops. Hence
$
  c_{1,\infty}\left(\Gam\right)\le\gamma^+\left(\Gam\right).
$

To bound $\gamma^+\left(\Gam\right)$, choose each vertex independently with probability
$
  p:=\frac{\log\left(d+1\right)}{d+1}.
$
Let $X$ be the set of chosen vertices and let
$$
  U:=\left\{v\in V\left(\Gam\right):N^-\left[v\right]\cap X=\varnothing\right\},
$$
where $N^-\left[v\right]$ is the closed in-neighbourhood of $v$. Since
$\left|N^-\left[v\right]\right|\ge d+1$,
$$
  \Pr\left(v\in U\right)\le\left(1-p\right)^{d+1}
  \le e^{-p\left(d+1\right)}
  =\frac1{d+1}.
$$
Therefore
$$
  \E\left(\left|X\right|+\left|U\right|\right)
  \le
  \frac{n\log\left(d+1\right)}{d+1}+\frac n{d+1}
  =
  \frac{n\left(1+\log\left(d+1\right)\right)}{d+1}.
$$
For every realization, $X\cup U$ is out-dominating. Hence there exists a realization of the asserted size.
\end{proof}

\subsection{The Alon--Mehrabian examples}\label{sec:alon}

\begin{proof}[Proof of Theorem~\ref{thm:intro-alon}]
We use the following construction of Alon and Mehrabian. Let
$$
  F:=\mathbb F_{2^k},\qquad d:=\left|F\right|=2^k,
$$
and enumerate the elements of $F$ as $x_1,\dots,x_d$. After identifying the additive group of $F$ with $\ZZ_2^k$, set
$$
  G:=\ZZ_2\times F^{s+1}
  \cong\ZZ_2^{\,1+k\left(s+1\right)}.
$$
For $1\le i\le d$, define
$$
  e_i:=\left(1,x_i,x_i^3,\dots,x_i^{2s+1}\right)\in G,
$$
where the powers are taken in $F$, and define
$$
  S:=\left\{e_1,\dots,e_d\right\},
  \qquad
  \Gam:=\Cay\left(G,S\right).
$$
Thus $\Gam$ has degree $d$ and
$$
  n=\left|G\right|=2^{\,1+k\left(s+1\right)}=2d^{s+1}.
$$
For each fixed $s$ and all sufficiently large $k$, Alon and Mehrabian proved
\cite[Lemma 2 and proof of Theorem 1]{AlonMehrabian} that $\Gam$ is connected
and that its speed-$s$ cop number satisfies
$$
  c_{1,s}\left(\Gam\right)=\Omega_s\left(d^s\right)
  =\Omega_s\left(n^{\frac{s}{s+1}}\right).
$$

For the upper bound, which also applies to an unbounded-speed robber, let
$$
  \phi:G\longrightarrow F,\qquad
  \phi\left(a_0,a_1,\dots,a_{s+1}\right)=a_1,
$$
be projection onto the first $F$-coordinate, and define
$$
  D_0:=\ker\phi.
$$
Since $\phi$ is onto,
$$
  \left|D_0\right|=\frac{\left|G\right|}{\left|F\right|}=\frac nd.
$$
Place one cop at each vertex of $D_0$. An initial robber position in $D_0$ is
occupied by a cop. Otherwise the robber starts at a vertex $v\notin D_0$, and
$\phi\left(v\right)=x_i$ for some $i$. Since
$\phi\left(e_i\right)=x_i$ and the group has characteristic two,
$$
  \phi\left(v+e_i\right)=\phi\left(v\right)+\phi\left(e_i\right)=x_i+x_i=0.
$$
Hence $v+e_i\in D_0$. The cop at $v+e_i$ moves along the edge generated by $e_i$ to
$$
  \left(v+e_i\right)+e_i=v
$$
and captures the robber on the first move by the cops. Therefore
$$
  c_{1,\infty}\left(\Gam\right)\le \left|D_0\right|
  =\frac nd
  =2d^s
  =2^{\frac{1}{s+1}}n^{\frac{s}{s+1}}.
$$
Combining this estimate with the Alon--Mehrabian lower bound and
$c_{1,s}\left(\Gam\right)\le c_{1,\infty}\left(\Gam\right)$ proves both
asymptotic statements.
\end{proof}

\section{Quotient-sweeping}\label{sec:quotient-method}

\subsection{Overview of the method}
The quotient method addresses the sparse regime. A winning quotient strategy
with $k$ cops lifts to the original graph by occupying the corresponding full
fibres and requires $k$ times the cardinality of a fibre.

\begin{proposition}[Quotient lifting]\label{prop:lifting}
Let $\Gam=\Cay\left(G,S\right)$ be a strongly connected abelian Cayley
digraph, and let $\phi:G\longrightarrow P$ be a surjective homomorphism onto
a finite group. Define
$
 \overline\Gam:=\Cay\left(P,\phi\left(S\right)\setminus\left\{0\right\}\right).
$
Then
$$
  c_{1,\infty}\left(\Gam\right)
  \le \left|\ker\phi\right|\,c_{1,\infty}\left(\overline\Gam\right).
$$
The same inequality holds with $c_{1,s}$ in place of $c_{1,\infty}$ for every finite $s$.
\end{proposition}

\begin{proof}
For each cop in a winning quotient strategy, occupy the full fibre above its
position. If a quotient cop moves by $\phi\left(g\right)$,
with $g\in S$, translate every cop in its fibre by $g$; this maps the full
fibre bijectively onto the required new fibre. A robber path in $\Gam$
initially projects to a directed walk in $\overline\Gam$, after consecutive
equal images arising from generators in $\ker\phi$ are deleted. This walk has
no greater length than the original path and avoids every quotient vertex
occupied by a cop because its entire fibre is occupied. Repeatedly delete the
closed subwalk between two occurrences of the same quotient vertex. The
resulting directed path has the same endpoints, no greater length, and still
avoids the quotient cop set. It is therefore a legal quotient robber path; in
the finite-speed game its length is at most $s$, and in the unbounded-speed
game it remains finite. When the quotient strategy moves a cop onto the
robber's image, the lifted cops occupy the robber's entire fibre and capture
it. The lifted strategy uses
$\left|\ker\phi\right|$ times as many cops as the quotient strategy.
\end{proof}

\begin{remark}
The constructions used in Proposition~\ref{prop:sweep} and
Theorem~\ref{thm:intro-alon} are both applications of
Proposition~\ref{prop:lifting}. Proposition~\ref{prop:sweep} lifts, through the
cyclic quotient $\pi:G\to\ZZ_m$, the quotient strategy that occupies two
intervals of $Q$ consecutive residues in
$\Cay\left(\ZZ_m,\pi\left(S\right)\setminus\left\{0\right\}\right)$. This
strategy uses $2Q$ quotient cops, and each fibre has cardinality
$\left|\ker\pi\right|=\frac nm$. Similarly, the proof of
Theorem~\ref{thm:intro-alon} lifts, through $\phi:G\to F$, a winning strategy
for one cop on
$\Cay\left(F,F\setminus\left\{0\right\}\right)\cong K_d$, whose fibres have
cardinality $\left|\ker\phi\right|=\frac nd$. Both lifting arguments are
stated explicitly to describe the corresponding fibre configurations.
\end{remark}

For the character-theoretic application, we construct a cyclic quotient
$$
  \pi:G\to\mathbb Z_m
$$
in which every generator has bounded cyclic displacement. The full preimages of sufficiently
long intervals in $\mathbb Z_m$ then form barriers: a legal robber path cannot
cross such a barrier without entering an occupied fibre. One barrier remains
stationary while the other moves in a generator direction, confining the
robber to a shrinking interval.

Simultaneous control of all generator images follows from a character-theoretic
pigeonhole argument. For any subgroup $K<G$, applying this argument to
$$
  \widehat{G/K}=\operatorname{Hom}\left(G/K,\mathbb R/\mathbb Z\right)
$$
yields a nontrivial character whose values on the generators are uniformly
close to $0$. This character-selection argument originates with
Friedman--Murty--Tillich
\cite[Section 2]{FriedmanMurtyTillichSpectralAbelianCayley}, who used it to
bound the second eigenvalue. Here, rather than passing to a Rayleigh quotient,
we retain the image of the character,
which supplies the required cyclic quotient. The argument applies to $c_{1,\infty}$ because the barriers depend only on the
largest single-step displacement in the quotient, not on the total length of a
robber move. The pursuit is thereby reduced to a monotone
interval-contraction process on a cycle.

\subsection{Character selection}\label{sec:character}
We construct a cyclic quotient of $G$ in which every generator has bounded
cyclic displacement.

\begin{proposition}\label{prop:character}
Let $\Gam=\Cay\left(G,S\right)$ be a strongly connected abelian Cayley digraph, and let $K<G$ be a proper subgroup. Define
$$
  \overline G:=G/K,
  \qquad
  r_K:=\left|\left\{\left\{\overline g,-\overline g\right\}:g\in S,\,
  \overline g\ne0\text{ in }G/K\right\}\right|,
$$
where $\overline g:=g+K$. Define $\overline n:=\left|\overline G\right|$ and set
$$
  N_K:=\floor{\left(\overline n-1\right)^{\frac{1}{r_K}}}.
$$
Then there exist an integer $m>N_K$ and a surjective homomorphism
$
  \pi:G\to\ZZ_m
$
whose kernel contains $K$, such that for every $g\in S$, the least-absolute-value representative $a\left(g\right)\in\ZZ$ of $\pi\left(g\right)$ satisfies
$$
  \left|a\left(g\right)\right|\le\ceil{\frac m{N_K}}-1.
$$
Moreover, at least one generator has nonzero image under $\pi$.
\end{proposition}

\begin{proof}
Choose a generator image $\overline g_i$ from each nonzero inverse-pair class in $\overline G$, say $\overline g_1,\dots,\overline g_{r_K}$. Every generator image is $0$ or one of the $\overline g_i$ or its negative. Since $S$ generates $G$, the elements $\overline g_1,\dots,\overline g_{r_K}$ generate $\overline G$ as a group. In particular, $r_K\ge1$.

Let
$$
  \widehat{\overline G}:=\operatorname{Hom}\left(\overline G,\RR/\ZZ\right).
$$
For $x\in\RR/\ZZ$, write $\norm{x}$ for its distance from $0$; equivalently, if $\widetilde x\in\RR$ is any representative of $x$, then
$$
  \norm{x}:=\min_{k\in\ZZ}\left|\widetilde x-k\right|.
$$
Under the isomorphism
$$
  \RR/\ZZ\longrightarrow\left\{z\in\mathbb C:\left|z\right|=1\right\},
  \qquad x\longmapsto e^{2\pi i x},
$$
the elements of $\widehat{\overline G}$ correspond to the one-dimensional complex representations of $\overline G$. The standard character theorem for finite abelian groups states that there are exactly $\left|\overline G\right|$ such representations \cite[Corollary 11(1), Section 18.2]{DummitFooteAbstractAlgebra}; see also \cite[Section 4.1]{TaoVuAdditiveCombinatorics} for the formulation in terms of finite Fourier analysis and the dual group. Therefore
$$
  \left|\widehat{\overline G}\right|=\left|\overline G\right|=\overline n.
$$
Moreover, $N_K^{r_K}\le \overline n-1$, and hence $\overline n>N_K^{r_K}$.

Represent each copy of $\RR/\ZZ$ by the half-open fundamental domain $0\le x<1$, divide it into $N_K$ half-open intervals of length $\frac{1}{N_K}$, and take Cartesian products. This partitions $\left(\RR/\ZZ\right)^{r_K}$ into $N_K^{r_K}$ half-open cubes. Now consider the evaluation map
$$
  \Phi:\widehat{\overline G}\to\left(\RR/\ZZ\right)^{r_K},
  \qquad
  \Phi\left(\chi\right)=
  \left(\chi\left(\overline g_1\right),\dots,\chi\left(\overline g_{r_K}\right)\right).
$$
There are $\overline n$ characters but only $N_K^{r_K}$ cubes, so the pigeonhole principle gives two distinct characters $\chi_1,\chi_2$ whose evaluation vectors lie in the same cube. Set
$$
  \chi:=\chi_1-\chi_2.
$$
Because $\chi_1\ne\chi_2$, their difference $\chi$ is a nonzero character. For each $i$, the representatives of $\chi_1\left(\overline g_i\right)$ and $\chi_2\left(\overline g_i\right)$ lie in the same interval of length $\frac{1}{N_K}$. Their difference therefore has distance less than $\frac{1}{N_K}$ from $0$ in $\RR/\ZZ$. Thus
$$
  \norm{\chi\left(\overline g_i\right)}<\frac1{N_K}
  \qquad\left(1\le i\le r_K\right).
$$

Let $\chi\left(\overline G\right)$ have order $m$. The unique subgroup of
$\RR/\ZZ$ of order $m$ is
$$
  \frac{1}{m}\ZZ/\ZZ.
$$
This subgroup is canonically isomorphic to $\ZZ/m\ZZ$ via
$$
  k+m\ZZ\longmapsto \frac{k}{m}+\ZZ.
$$
Consequently, for every $\overline x\in\overline G$, there is a unique residue
$b\left(\overline x\right)\in\ZZ_m$ such that
$$
  \chi\left(\overline x\right)
  =\frac{j}{m}+\ZZ
$$
for any integer representative $j$ of $b\left(\overline x\right)$.
The map $b:\overline G\to\ZZ_m$ is a surjective homomorphism. Composing it
with the quotient map $G\to\overline G$ gives a surjective homomorphism
$$
  \pi:G\to\ZZ_m
$$
whose kernel contains $K$. For $g\in S$, let $a\left(g\right)$ be the
least-absolute-value integer representative of $\pi\left(g\right)$. Since
$\overline g$ is zero or belongs, up to sign, to one of the chosen
inverse-pair representatives, the preceding character estimate gives
$$
  \frac{\left|a\left(g\right)\right|}{m}
  =\norm{\chi\left(\overline g\right)}
  <\frac{1}{N_K}.
$$
Hence $\left|a\left(g\right)\right|<\frac{m}{N_K}$ for every $g\in S$.
Since $S$ generates $G$ and $\pi$ is nontrivial, at least one generator has nonzero image. Hence $1<\frac{m}{N_K}$, so $m>N_K$, and integrality gives
$$
  \left|a\left(g\right)\right|\le\ceil{\frac m{N_K}}-1.
$$
\end{proof}

\begin{remark}
The pigeonhole step is a simultaneous-approximation argument that produces a
cyclic quotient in which every generator has bounded cyclic displacement.
Varying $K$ can
reduce the number $r_K$ of generator-image classes while retaining a large
quotient. The subgroup $H$ generated by the involutory generators provides a
uniform choice when $H<G$, although it need not be optimal.
\end{remark}

\begin{remark}\label{rem:rH}
For the choice $K=H$, one has
$$
  r_H\le\left|S\right|-\left|T\right|,
$$
and if $S=-S$, then
$$
  r_H\le\frac{\left|S\right|-\left|T\right|}{2},
$$
since the non-involutory generators then form inverse pairs, whose images in
$G/H$ remain inverse to one another. Only these upper bounds are used below,
and a smaller value of $r_H$ only strengthens
Theorem~\ref{thm:intro-quotient}. Both inequalities may be strict, because
distinct generators may share an image in $G/H$. For instance, let
$$
  G=\ZZ_4\times\ZZ_2,\qquad
  S=\left\{\left(0,1\right),\left(1,0\right),\left(3,0\right),
  \left(1,1\right),\left(3,1\right)\right\}.
$$
Then $S=-S$, $T=\left\{\left(0,1\right)\right\}$,
$H=\left\{0\right\}\times\ZZ_2$, and $G/H\cong\ZZ_4$, in which the four
non-involutory generators have images $1,3,1,3$. Hence $r_H=1$, whereas
$\frac{\left|S\right|-\left|T\right|}{2}=2$. Thus symmetry reduces the general
upper bound on $r_H$ by a factor of two, which yields the stronger uniform
undirected exponent; identifications in the quotient may reduce $r_H$
further.
\end{remark}

\subsection{Sweeping in a cyclic quotient}\label{sec:sweep}

Mehrabian lifted quotient strategies by replacing each quotient cop with a
full fibre of cops \cite[Theorem 5.1]{MehrabianFastRobberExpandersRandomGraphs}.
His application concerned Cartesian products, which carry a canonical
projection. Here a cyclic quotient is constructed for abelian Cayley digraphs
without an assumed product decomposition. The resulting pursuit argument is
the following.

\begin{proposition}\label{prop:sweep}
Let $\Gam=\Cay\left(G,S\right)$ be a strongly connected abelian Cayley digraph on $n$ vertices, and suppose there is a surjective homomorphism
$
  \pi:G\to\ZZ_m.
$
For each $g\in S$, choose an integer $a\left(g\right)\in\ZZ$ representing the residue $\pi\left(g\right)$, with $\left|a\left(g\right)\right|$ as small as possible. Assume that
$$
  \left|a\left(g\right)\right|\le Q
$$
for some integer $Q\ge1$, and that
$$
  m\ge2Q+2.
$$
Then, for every finite $s$,
$$
  c_{1,s}\left(\Gam\right)\le c_{1,\infty}\left(\Gam\right)
  \le\frac{2Qn}{m}.
$$
\end{proposition}

\begin{proof}
Choose two disjoint intervals of $Q$ consecutive residues in $\ZZ_m$, separated
in each cyclic direction by at least one unoccupied residue. The complement of
their union consists of two nonempty intervals, which we call \emph{gaps}.
Occupy the full preimage of each selected interval with cops. Every fibre of
$\pi$ has cardinality $\frac{n}{m}$, so the resulting configuration uses
$$
  2Q\frac nm
$$
cops.

Fix a cyclic orientation of $\ZZ_m$. After the robber chooses its initial
vertex $x_0$, its image $\pi\left(x_0\right)$ lies in one of the two unoccupied
gaps; denote this gap by $J\left(0\right)$, and write $J\left(t\right)$ for the
gap containing the robber's image after $t$ uncaptured rounds. Let $I_L$ and
$I_R$ denote the two occupied intervals, labelled so that, in the chosen
orientation, $I_L$ immediately precedes $J\left(0\right)$ and $I_R$
immediately follows it; we call these the \emph{left} and \emph{right}
barriers. Exactly one barrier moves throughout the strategy, so
$J\left(t\right)$ is bounded
by $I_L$ and $I_R$ for every $t$.

Since $S$ generates $G$ and $\pi$ is onto, some generator has nonzero image. Choose $g_*\in S$ with
$$
  a:=a\left(g_*\right)\in\ZZ\setminus\left\{0\right\}.
$$
Thus $\pi\left(g_*\right)$ is a residue modulo $m$, whereas $a$ is its least-absolute-value integer representative. In particular, $a$ may be positive or negative, and
$$
  \pi\left(x+g_*\right)=\pi\left(x\right)+a\pmod m
$$
for every $x\in G$. If $a>0$, the left barrier is translated by $g_*$, while
the right barrier remains fixed. If $a<0$, the right barrier is translated by
$g_*$, while the left barrier remains fixed. In either case, the moving
barrier advances $\left|a\right|$ residues into the robber's gap.

Let $I\subseteq\ZZ_m$ be the interval occupied by the moving barrier before
translation. Its occupied vertex set is $\pi^{-1}\left(I\right)$. Since
translation by $g_*$ is a bijection of $G$ and
$\pi\left(g_*\right)=a\pmod m$,
$$
  \pi^{-1}\left(I\right)+g_*=\pi^{-1}\left(I+a\right).
$$
Indeed, if $y\in\pi^{-1}\left(I+a\right)$, then $y-g_*\in\pi^{-1}\left(I\right)$, so the cop formerly at $y-g_*$ moves to $y$. Thus every vertex in every fibre above the translated interval $I+a$ is occupied.

If the translated barrier contains the robber's residue, then a cop occupies the robber's vertex and capture occurs. Otherwise the robber remains in a gap $J\left(t+1\right)$ bounded by the translated barrier and the stationary barrier, and
$$
  \left|J\left(t+1\right)\right|=\left|J\left(t\right)\right|-\left|a\right|.
$$

No legal robber path can move from this gap to the other complementary gap.
To verify this assertion, consider a legal robber path
$$
  x_0,x_1,\dots,x_\ell.
$$
Any directed step from one complementary gap to the other across a barrier of
$Q$ consecutive occupied residues requires cyclic displacement at least
$Q+1$. Each projected generator has least-absolute-value displacement at most
$Q$, so no individual robber step can cross either barrier. Since an arbitrary
robber move is a sequence of generator steps, the robber can leave
$J\left(t+1\right)$ only by visiting an occupied fibre, which is forbidden.

Thus the robber remains in a gap whose size decreases by $\left|a\right|\ge1$ in every uncaptured round. If $\left|J\left(t\right)\right|\le\left|a\right|$, then the next translation of the moving barrier contains every residue of the gap; because the barrier occupies full fibres, a cop occupies the robber's vertex. Hence capture occurs after finitely many rounds.
\end{proof}

\begin{remark}
It suffices for the interval length to equal the largest \emph{single-step}
displacement in the quotient, because the robber may not pass through occupied
vertices. The interval length is independent of the total length of the
robber's move.
\end{remark}

\begin{remark}
For the directed cycle $\Cay\left(\ZZ_n,\left\{1\right\}\right)$ with $n\ge4$, take $K=\left\{0\right\}$; then $r_K=1$ and $N_K=n-1$, and the construction yields $c_{1,\infty}\le2$. For $n=3$, two cops suffice when placed on consecutive vertices. Conversely, one cop does not suffice for any $n\ge3$: the robber initially chooses a vertex at directed distance two from the cop and advances one step after each move of the cop. Thus the directed distance from the cop to the robber is again two after every robber move. Consequently,
$$
  c_{1,\infty}\left(\Cay\left(\ZZ_n,\left\{1\right\}\right)\right)=2
  \qquad\left(n\ge3\right).
$$
\end{remark}

\subsection{The optimized quotient bound}\label{sec:main}

The character-selection lemma and the cyclic sweeping proposition together
prove the optimized quotient bound.

\begin{proof}[Proof of Theorem~\ref{thm:intro-quotient}]
Use the notation in the theorem.
Fix a proper subgroup $K<G$. If $N_K\le 2$, then
$\frac{2n}{N_K}\ge n$, so this estimate is no stronger than the elementary
bound. We may therefore assume that $N_K\ge 3$.

Apply Proposition~\ref{prop:character}. We obtain a surjective homomorphism
$$
  \pi:G\to \ZZ_m
$$
for some $m>N_K$ such that every generator image has least-absolute-value representative bounded by
$$
  Q:=\ceil{\frac m{N_K}}-1.
$$
Thus $1\le Q<\frac{m}{N_K}$. Since $m>N_K\ge3$, we have $m\ge4$, and hence $Q<\frac{m}{3}$. Therefore
$$
  2Q+2
  \le 2\ceil{\frac m3}
  \le \frac{2m+4}{3}
  \le m.
$$
Thus Proposition~\ref{prop:sweep} applies and yields
$$
  c_{1,\infty}\left(\Gam\right)\le \frac{2Qn}{m}.
$$
Since $Q<\frac{m}{N_K}$, we obtain
$$
  c_{1,\infty}\left(\Gam\right)< \frac{2n}{N_K}.
$$
This holds for every proper subgroup $K$ with $N_K\ge3$. If $N_*\le2$, then
$\frac{2n}{N_*}\ge n$, while placing one cop at each vertex gives
$c_{1,\infty}\left(\Gam\right)\le n$. Choosing $K$ with $N_K=N_*$ and combining
with Proposition~\ref{thm:domination} proves the stated minimum.
\end{proof}

\section{Consequences and concluding remarks}\label{sec:consequences}

\subsection{Bounded-degree and sparse consequences}

\begin{proof}[Proof of Theorem~\ref{thm:intro-bounded}]
Use the notation in the theorem.
Let $T,t,H$ be as in Section~\ref{sec:prelim}. Since $H$ is generated by $t\le D$ involutions,
$$
  \left|H\right|\le 2^t\le 2^D.
$$
If $H=G$, then $n\le2^D$, and the result follows after enlarging the implicit
constant depending on $D$. Otherwise take $K=H$ in
Theorem~\ref{thm:intro-quotient}, and write $r_H,N_H$ for the corresponding
parameters. Then
$$
  \left|G/H\right|\ge \frac{n}{2^D},
  \qquad
  r_H\le \left|S\right|-t\le D.
$$
Define
$$
  N_H
  :=\floor{\left(\left|G/H\right|-1\right)^{\frac{1}{r_H}}}.
$$
For all sufficiently large $n$,
$$
  N_*\ge N_H=\Omega_D\left(n^{\frac{1}{D}}\right).
$$
Theorem~\ref{thm:intro-quotient} gives the directed estimate.

If $S=-S$ and $q:=\floor{\frac{D}{2}}$, then the non-involutory generators form inverse pairs before passing to the quotient, and hence
$$
  r_H\le\frac{\left|S\right|-t}{2}\le q.
$$
Thus
$$
  N_*\ge N_H=\Omega_D\left(n^{\frac{1}{q}}\right),
$$
and Theorem~\ref{thm:intro-quotient} gives the claimed undirected estimate.
The finitely many remaining orders are absorbed into the constant depending on $D$. The optimality assertion follows from the Cartesian powers of cycles discussed in Section~\ref{sec:conclusion}.
\end{proof}

The quotient-sweeping estimate is sublinear whenever the degree grows more
slowly than $\log n$.

\begin{corollary}[Logarithmic degree]\label{cor:logdeg}
Fix $\kappa>0$. Suppose $\Gam=\Cay\left(G,S\right)$ is a strongly connected abelian Cayley digraph on $n$ vertices with out-degree
$
  d\le \frac{\log n}{\kappa\log\log n}
$
for all sufficiently large $n$. Then for every $s\ge 1$,
$$
  c_{1,s}\left(\Gam\right)\le c_{1,\infty}\left(\Gam\right)
  = O\!\left(\frac{n}{\left(\log n\right)^{\kappa}}\right).
$$
If $S=-S$, the denominator improves to $\left(\log n\right)^{2\kappa}$.
\end{corollary}

\begin{proof}
Let $T,t,H$ be as in Section~\ref{sec:prelim}. We have $\left|H\right|\le2^d$. If $H=G$, then $n\le2^d$, which contradicts the assumed upper bound on $d$ for all sufficiently large $n$. Hence $G/H$ is nontrivial. Take $K=H$ in Theorem~\ref{thm:intro-quotient}, and write $r_H,N_H$ for the corresponding parameters. Then
$$
  \left|G/H\right|\ge\frac{n}{2^d},
  \qquad r_H\le d.
$$
For sufficiently large $n$, this implies
$$
  N_H:=\floor{\left(\left|G/H\right|-1\right)^{\frac{1}{r_H}}}
  \ge \left(\frac{n}{2^{d+1}}\right)^{\frac{1}{d}}-1
  \ge \frac{n^{\frac{1}{d}}}{8}.
$$
Now
$$
  n^{\frac{1}{d}}=\exp\!\left(\frac{\log n}{d}\right)
  \ge \exp\!\left(\kappa\log\log n\right)=\left(\log n\right)^{\kappa}.
$$
Since $N_*\ge N_H$, Theorem~\ref{thm:intro-quotient} gives the absolute estimate
$$
  c_{1,\infty}\left(\Gam\right)\le
  16\frac{n}{\left(\log n\right)^{\kappa}}.
$$
If $S=-S$, then $r_H\le \frac d2$, and the same proof gives the stated
denominator $\left(\log n\right)^{2\kappa}$; one may take $64$ as an absolute
constant.
\end{proof}

Combining the quotient estimate in the sparse range with domination in the
dense range removes the degree hypothesis.

\begin{proof}[Proof of Theorem~\ref{thm:intro-degree-free}]
Assume that $n$ is sufficiently large, as permitted by the asymptotic statement, and define
$$
  L:=\log n,\qquad \ell:=\log L,\qquad d:=\left|S\right|.
$$
Suppose first that $d\ge \frac{L}{\ell}$. The function
$x\longmapsto\frac{1+\log x}{x}$
is decreasing for $x\ge1$, so Proposition~\ref{thm:domination} gives
$$
  c_{1,\infty}\left(\Gam\right)
  \le n\frac{1+\log\left(d+1\right)}{d+1}
  =O\!\left(\frac{n\ell^2}{L}\right).
$$

Now suppose that $d<\frac{L}{\ell}$. Let $T,t,H$ be as in Section~\ref{sec:prelim}. Since $\left|H\right|\le2^t$, the equality $H=G$ would imply
$$
  L=\log n\le t\log2\le d\log2<\frac{L\log2}{\ell},
$$
which is impossible for large $n$. Thus $H<G$. Since $S$ generates $G$, some
generator lies outside $H$; every element of $T$ lies in $H$, so
$d-t\ge1$. Define $\overline n:=\left|G/H\right|$, and let $r_H,N_H$ be the
parameters of Theorem~\ref{thm:intro-quotient} for $K=H$. Since
$\overline n\ge \frac{n}{2^t}$ and $\overline n\ge2$,
$$
  \log\left(\overline n-1\right)\ge L-t\log2-\log2.
$$
Also $r_H\le d-t$. For large $n$, one has $\ell\ge2\log2$, and
$$
\begin{aligned}
 &L-\left(t+1\right)\log2-\left(d-t\right)\left(\ell-\log2\right)\\
 &\qquad=\left(L-d\ell\right)+\left(d-1\right)\log2+t\left(\ell-2\log2\right)\ge0.
\end{aligned}
$$
Consequently,
$$
  \log\!\left(\left(\overline n-1\right)^{\frac{1}{r_H}}\right)
  \ge\frac{L-t\log2-\log2}{d-t}
  \ge\ell-\log2.
$$
It follows that $N_H=\Omega\left(L\right)$, and hence $N_*\ge N_H=\Omega\left(\log n\right)$. Theorem~\ref{thm:intro-quotient} gives
$$
  c_{1,\infty}\left(\Gam\right)=O\!\left(\frac n{\log n}\right),
$$
which is stronger than required in this range.
\end{proof}

\begin{remark}
  For each fixed finite $s$, Frieze--Krivelevich--Loh proved that every
  undirected $n$-vertex graph $\Gam$ satisfies
  \cite[Theorem 1.3]{FriezeKrivelevichLohVariationsOnCopsRobbers}
$$
  c_{1,s}\left(\Gam\right)
  \le
  \frac{n}{\exp\!\left(\left(1-o\left(1\right)\right)
  \sqrt{\log\!\left(1+\frac{1}{s}\right)\log n}\right)}.
$$
Thus, for fixed $s$, their bound is asymptotically stronger than our degree-independent
bound for undirected abelian Cayley graphs. When the degree is bounded, our
quotient estimate instead improves the trivial linear bound by a polynomial
factor in $n$. Moreover,
Theorems~\ref{thm:intro-quotient} and~\ref{thm:intro-degree-free} are uniform
in the robber's speed, apply when $s=\infty$, and include directed abelian
Cayley graphs. To the best of our knowledge, no general sublinear fast-robber
bound for strongly connected digraphs was previously available.
\end{remark}

\subsection{Examples, lower bounds, and questions}\label{sec:conclusion}

The following isoperimetric lower bound complements the quotient upper bound in
the undirected case. Mehrabian's proof uses undirected components and paths;
accordingly, no directed analogue is asserted here.

\begin{proposition}[Mehrabian]\label{prop:isoperimetric-lower}
Let $\Gam$ be a connected graph on $n$ vertices with maximum degree $\Delta$, and let $\iota_v\left(\Gam\right)$ denote its vertex-isoperimetric number. Then
$
  c_{1,\infty}\left(\Gam\right)
  \ge \frac{\iota_v\left(\Gam\right)n}{4\Delta}.
$
In particular, if $\Gam=\Cay\left(G,S\right)$ has degree $d$, then
$
  c_{1,\infty}\left(\Gam\right)
  \ge \frac{\iota_v\left(\Gam\right)n}{4d}.
$
\end{proposition}

\begin{proof}
The assertion is \cite[Theorem 3.3(c)]{MehrabianFastRobberExpandersRandomGraphs}.
\end{proof}

\begin{remark}\label{rem:near-tight}
Theorem~\ref{thm:intro-degree-free} is close to optimal within the class of
abelian Cayley graphs. By the Alon--Roichman theorem
\cite{AlonRoichmanRandomCayley}, for every finite abelian group $G$ of order
$n$ and every $\varepsilon>0$ there is a symmetric set
$S\subseteq G\setminus\left\{0\right\}$ with
$\left|S\right|=O_\varepsilon\left(\log n\right)$ such that
$\Gam:=\Cay\left(G,S\right)$ has normalized second eigenvalue at most
$\varepsilon$. Taking, for example, $\varepsilon=\frac12$, Tanner's inequality
implies that its vertex-isoperimetric number is bounded below by a positive
absolute constant. Proposition~\ref{prop:isoperimetric-lower} gives
$$
  c_{1,\infty}\left(\Gam\right)
  \ge\frac{\iota_v\left(\Gam\right)n}{4\left|S\right|}
  =\Omega\!\left(\frac n{\log n}\right)
$$
for these graphs. Hence no bound of the form
$o\!\left(\frac n{\log n}\right)$ can hold for all abelian Cayley graphs.
Thus the upper bound in Theorem~\ref{thm:intro-degree-free} differs from this
lower bound by a factor of at most $\left(\log\log n\right)^2$.
\end{remark}

The quotient estimate is sharp in order on Cartesian powers of cycles. Let
$$
  \Gam=C_m\square\cdots\square C_m
  =\Cay\left(\ZZ_m^k,\left\{\pm e_1,\dots,\pm e_k\right\}\right),
  \qquad m\ge4.
$$
Projection onto the first coordinate has quotient order $m$ and maps every generator to $0$ or $\pm1$. Proposition~\ref{prop:sweep} therefore gives
$$
  c_{1,\infty}\left(\Gam\right)\le 2m^{k-1}.
$$
The resulting estimate coincides with the Cartesian-product upper bound of Mehrabian
\cite[Theorem 5.1(a)]{MehrabianFastRobberExpandersRandomGraphs}. When $m$ is even, his matching lower estimate establishes
$$
  \frac{m^{k-1}}{2k^2}
  \le c_{1,\infty}\left(C_m^{\square k}\right)
  \le 2m^{k-1},
$$
so the sweeping estimate is sharp up to a constant factor for every fixed $k$
along even cycle lengths
\cite[Theorem 5.1(c)]{MehrabianFastRobberExpandersRandomGraphs}. For $k=2$, Kinnersley and Townsend sharpened this to
$$
  2m-24\le c_{1,\infty}\left(C_m\square C_m\right)\le2m
  \qquad\left(m\ge18\right),
$$
so the sweeping construction has the correct leading constant in dimension two
\cite[Theorem 2.7]{KinnersleyTownsendInfiniteGrid}.

These examples also establish the optimality of the undirected bounded-degree exponent in Theorem~\ref{thm:intro-bounded}. Fix $D\ge2$, define $q:=\floor{\frac{D}{2}}$, and set $k=q$. The graphs $C_m^{\square q}$ have degree $2q\le D$, order $n=m^q$, and, for even $m$,
$$
  c_{1,\infty}\left(C_m^{\square q}\right)
  =\Omega_q\left(m^{q-1}\right)
  =\Omega_D\left(n^{1-\frac{1}{q}}\right).
$$
Thus the optimal uniform order is constant for $D=2,3$, is $n^{\frac{1}{2}}$ for $D=4,5$, and is $n^{\frac{2}{3}}$ for $D=6,7$. Moreover, the Kinnersley--Townsend lower bound shows that the factor $2$ in Proposition~\ref{prop:sweep} cannot be replaced by any smaller universal constant.

Optimization over the subgroup can be necessary to obtain the correct order
of magnitude. For example, let $M\ge4$ be even and consider
$$
  \Gam=C_M\square C_4
  =\Cay\left(\ZZ_M\times\ZZ_4,\left\{\pm e_1,\pm e_2\right\}\right).
$$
Here $H=\left\{0\right\}$, and the parameter associated with $K=H$ is
$$
  N_H=\floor{\left(4M-1\right)^{\frac{1}{2}}}=\Theta\left(\sqrt M\right),
$$
which yields the weaker bound $O\left(\sqrt M\right)$. By contrast, projection
onto the first coordinate has fibres of size $4$ and maps every generator to
$0$ or $\pm1$, so Proposition~\ref{prop:sweep} gives
$$
  c_{1,\infty}\left(C_M\square C_4\right)\le8.
$$
The same example shows that $N_H$ need not be comparable with the inverse
isoperimetric scale. Letting $A$ be the union of $M/2$ consecutive
$C_4$-fibres gives
$$
  \iota_v\left(C_M\square C_4\right)\le\frac4M,
$$
whereas $N_H=\Theta\left(\sqrt M\right)$.

For a fixed subgroup $K$, every cyclic quotient arising from a character of
$G/K$ has order at most $\operatorname{exp}\left(G/K\right)$. Since any
nontrivial cyclic barrier configuration occupies at least two full fibres,
every such configuration uses at least
$\frac{2n}{\operatorname{exp}\left(G/K\right)}$ cops. When $G$ has exponent two, every
cyclic quotient has order at most two, and the cyclic-sweep estimate is no
stronger than the elementary bound. The noncyclic quotient and
dominating-subgroup construction in Theorem~\ref{thm:intro-alon} may
nevertheless yield a nontrivial bound.

\begin{question}
Fix $d$. Does every connected undirected $d$-regular abelian Cayley graph $\Gam$ on $n$ vertices satisfy
$$
  c_{1,\infty}\left(\Gam\right)=\Theta_d\left(n\iota_v\left(\Gam\right)\right)?
$$
The lower bound is Proposition~\ref{prop:isoperimetric-lower}. In view of Theorem~\ref{thm:intro-quotient}, the upper bound would follow from a comparison of the form
$$
  N_*=\Omega_d\!\left(\frac{1}{\iota_v\left(\Gam\right)}\right).
$$
\end{question}

\begin{question}\label{q:degree-reduction-fast}
Is there a degree reduction for the fast-robber game? More precisely, are there a
function $f$ and a constant $C\ge1$ such that every graph $\Gam$ on $n$
vertices with maximum degree $\Delta$ admits a subcubic graph $H$ on at most
$f\left(\Delta\right)n$ vertices with
$$
  c_{1,Cs}\left(H\right)\ge c_{1,s}\left(\Gam\right)
  \qquad\left(s\ge1\right)?
$$
The analogous problem for digraphs, based on
\cite[Theorem 6]{HosseiniMoharGonzalezBoundedDegree}, also remains open.
\end{question}

When regarded as symmetric digraphs, the Alon--Mehrabian examples have cop
number of order $n^{\frac{s}{s+1}}$ for both a speed-$s$ robber and an
unbounded-speed robber.
For the ordinary directed game, Bradshaw--Hosseini--Turcotte constructed
abelian Cayley digraphs with cop number of order $\sqrt n$
\cite[Theorem 5.2]{BradshawHosseiniTurcotteCopsRobbersDirUnditAbelianCayley}.

\end{document}